\documentclass[11pt]{amsart}
\usepackage{latexsym}
\usepackage{amsfonts,amsmath,amssymb,stmaryrd,amscd,amsxtra,amsthm,verbatim,calc}
\usepackage[all]{xy}
\RequirePackage[dvipsnames,usenames]{color}

\usepackage{xypic}
\usepackage{eucal}

\usepackage{amsmath}
\usepackage{amstext}
\usepackage{amssymb}
\usepackage{amsthm}
\usepackage{amscd}

\usepackage{pst-plot}

\usepackage{lscape}
\usepackage{longtable}

\usepackage{epsfig}
\let\et=\etexdraw
\def\etexdraw{\drawbb\et}

\theoremstyle{plain}
\newtheorem{thm}{Theorem}[section]
\newtheorem{thm*}{Theorem}
\newtheorem{lem}[thm]{Lemma}
\newtheorem{prop}[thm]{Proposition}
\newtheorem{prop*}[thm*]{Proposition}
\newtheorem{cor}[thm]{Corollary}

\theoremstyle{definition}
\newtheorem{defi}[thm]{Definition}

\newtheorem{ex}[thm]{Example}
\newtheorem{qu}[thm]{Question}

\theoremstyle{remark}

\newtheorem{rmk}[thm]{Remark}

\DeclareMathOperator{\Coker}{Coker}
\DeclareMathOperator{\Image}{Im}

\DeclareMathOperator{\Hom}{Hom}
\DeclareMathOperator{\Ass}{Ass}

\DeclareMathOperator{\Ext}{Ext}

\DeclareMathOperator{\Ann}{ann}

\DeclareMathOperator{\HH}{H}

\DeclareMathOperator{\Supp}{Supp}
\DeclareMathOperator{\Spec}{Spec}
\DeclareMathOperator{\Res}{Res}

\DeclareMathOperator{\fa}{\mathfrak{a}}
\DeclareMathOperator{\fm}{\mathfrak{m}}

\begin{document}

\title{The support of local cohomology modules}

\author{Mordechai Katzman}
\address{Department of Pure Mathematics,
University of Sheffield, Hicks Building, Sheffield S3 7RH, United Kingdom}
\email{M.Katzman@sheffield.ac.uk}

\author{Wenliang Zhang}
\address{Department of Mathematics, Statistics, and Computer Science, University of Illinois at Chicago, 851 S. Morgan Street, Chicago, IL 60607-7045}
\email{wlzhang@uic.edu}

\thanks{M.K.  gratefully acknowledges support from EPSRC grant EP/J005436/1. W.Z. is partially supported by NSF grants DMS \#1405602/\#1606414.}


\begin{abstract}
We describe the support of $F$-finite $F$-modules  over  polynomial rings $R$ of prime characteristic. Our description yields an algorithm to compute the support of such modules; the complexity of our algorithm is also analyzed. To the best of our knowledge, this is the first algorithm to avoid extensive use of Gr\"obner bases and hence of substantial practical value.
We also use the idea behind this algorithm to prove that the support of $H^j_I(S)$ is Zariski closed for each ideal $I$ of $S$ where $R$ is noetherian commutative ring of prime characteristic with finitely many isolated singular points and $S=R/gR$ ($g\in R$).
\end{abstract}

\maketitle

\section{Introduction}\label{Section: Introduction}

Local cohomology is a powerful tool introduced by  Alexander Grothendieck in the 1960's (\cite{HartshorneLocalCohomology}) and it has since yielded many geometric and algebraic insights. From an algebraic point of view, given an ideal $I$ in a commutative ring $R$,
local cohomology modules $\HH_I^{i}(-)$ ($i\geq 0$) arise as right-derived functors of the torsion functor on $R$-modules given by
$\Gamma_I(M)=\{ a\in M \,|\, I^k a=0 \text{ for some } k\geq 0\}$. A central question in the theory of local cohomology is to determine for which values of $i$ does the local cohomology module $\HH_I^{i}(M)$ vanish. This question is both useful and difficult even in the case where $R$ is a regular local ring and $M=R$, and this case has been studied intensely since the introduction of local cohomology (e.g., cf.~\cite{HartshorneCohomologicalDimension}, \cite{PeskineSzpiroDimensionProjective} and \cite{OgusLocalCohomologicalDimension}).

The aim of this paper is to describe the support of local cohomology modules in prime characteristic.
Specifically, we first study the support of $F$-finite $F$-modules over polynomial rings
$R$ and show a computationally feasible
method for computing these without the need to compute generating roots. To the best of our knowledge, this is the first computationally feasible algorithm for  calculating the support of these modules in prime characteristic.
We then apply this to the calculation of supports of local cohomology modules and
of iterated local cohomology modules $\HH^{i_1}_{I_1}\left( \HH^{i_2}_{I_2}\left(\dots \HH^{i_n}_{I_n}(R) \dots \right) \right)$ thus, for example,
giving an effective method for determining the vanishing of Lyubeznik numbers.

Our methods are interesting both from theoretical and practical points of view.
A careful analysis of the algorithms resulting from these methods (see Section \ref{Section: bounds on degrees} below) shows that
\begin{itemize}
\item [(a)] the degrees of the polynomials appearing in the calculations have a low upper bound, and, furthermore,
\item [(b)] when the method is applied to the calculation of supports of local cohomology modules, if the input is given by polynomials with integer coefficients, then the calculation of supports
modulo different primes $p$ involves polynomials \emph{whose degrees can be bounded from above by a constant times $p$, that constant being independent of $p$.}
\end{itemize}

In \cite{LyubeznikFModulesApplicationsToLocalCohomology}) Gennady Lyubeznik described an algorithm for computing the support of $F$-finite $F$-modules. That algorithm requires the calculation for roots of these modules, and this relies on the repeated calculation of Grobner bases; these are often too complex to be computed in practice.

Our algorithm consists of an iterative procedure (as described in section \ref{Section: Calculation of supports}) which produces a quotient of a finite-rank free module with the same support as the given $F$-finite $F$-module. To find the support itself one needs to find a presentation for this submodule as a cokernel of a matrix: the support is then defined by the ideal of maximal minors of that matrix.

Crucially, the iterative procedure above \emph{does not require the calculation of Gr\"{o}bner bases,} and consists  essentially of matrix multiplications together with the listing of terms of polynomials whose  degrees are bounded by a constant (independent of $p$) times $p$. The final step of the algorithm, finding a presentation of a finitely generated module, requires the calculation of one module of syzygies, hence the calculation of \emph{one} Gr\"ober basis.

It is this that makes our algorithm a practical tool for computing supports of $F$-finite $F$-modules.\footnote{The various algorithms in this paper have been incorporated in the ``FSing'' package of Macaulay 2\cite{Macaulay}.}

\medskip
The reason why we are able to compute and analyze in characteristic $p$ the support of $F$-finite $F$-modules is the existence of
the  \emph{$e$th iterated Frobenius endomorphism} $f^e: R \rightarrow R$,
taking $a\in R$ to $a^{p^e}$ ($e\geq 0$).
The usefulness of these  lies in the fact that given an $R$-module $M$, we may endow it with a new $R$-module structure via $f^e$:
let  $F^e_* M$ denote the additive Abelian group $M$
denoting its elements $\{ F^e_* m \,|\, m\in M \}$, and endow $F^e_* M$  with the $R$-module structure  is given by  $a F_*^e m = F^e_* a^{p^e} m$
for all $a\in R$ and $m\in M$.

This also allows us to define  the \emph{$e$th Frobenius functors} from the category of $R$-modules to itself given by
$F_R^e(M)=F_*^e R \otimes_R M$ and viewing this as a $R$-module via the identification of $F_*^e R$ with $R$: the resulting $R$-module structure
on $F_R^e(M)$ satisfies
$a(F_*^e b \otimes m)=F_*^e a b \otimes m$ and  $F_*^e a^p b \otimes m=F_*^e b \otimes a m$ for all $a,b\in R$ and $m\in M$.

We will be interested in this construction mainly for regular rings and henceforth in this paper $R$ will denote a regular ring of characteristic $p>0$.

Recall that an $F$-finite $F_R$-module $\mathcal{M}$
is an $R$-module obtained as
a direct limit of a direct limit system of the form
$$
M \xrightarrow{U} F^1_R(M) \xrightarrow{F^1_R(U)} F^2_R(M) \xrightarrow{F^2_R(U)}  \dots
$$
where $M$ is a finitely generated module and $U$ is an $R$-linear map (cf.~\cite{LyubeznikFModulesApplicationsToLocalCohomology}).
The main interest in $F$-finite $F$-modules follows from the fact that local cohomology modules are  $F$-finite $F$-modules, as we now explain.

The \emph{$j$th local cohomology module of $M$
with support on an ideal $I\subset R$} is defined as
\begin{equation}\label{eqn1}
\HH^j_I(M)=
\lim_{\genfrac{}{}{0pt}{}{\rightarrow}{e}} \Ext^{j}_R ( R/I^{[p^e]}, M)
\end{equation}
where maps in the direct limit system are induced by the surjections $R/I^{[p^{e+1}]} \rightarrow R/I^{[p^{e}]}$.
If we apply this with $M=R$,
we obtain
\begin{eqnarray*}
\HH^j_I(R)&=&\lim_{\genfrac{}{}{0pt}{}{\rightarrow}{e}} \Ext^{j}_R ( R/I^{[p^e]}, R)\\
& \cong& \lim_{\genfrac{}{}{0pt}{}{\rightarrow}{e}} \Ext^{j}_R ( F_R^e(R/I), F_R^e R)\\
&  \cong& \lim_{\genfrac{}{}{0pt}{}{\rightarrow}{e}} F_R^e \left(\Ext^{j}_R ( R/I, R) \right)
\end{eqnarray*}
where we use the facts
that $F_R^e(R)\cong R$,  $F_R^e(R/I)\cong R/I^{[p^e]}$,
and that, since $R$ is regular, the Frobenius functor $F^e_R(-)$ is exact and thus commutes with the
computation of cohomology. This shows that
$\HH^j_I(R)$ are $F$-finite $F$-modules, and we may apply
our $F$-finite $F$-module machinery to them.

Finally, in section \ref{Section: Local cohomology of hypersurfaces} we turn our attention to
hypersurfaces and describe the support of their local cohomology modules, which turn out to be closed.\footnote{The fact that the support is closed was simultaneously and independently also discovered by Mel Hochster and Luis N{\'u}{\~n}ez-Betancourt in
\cite{HochsterNunezHypersurfacesAndFFRT} using a different method.}
Given a fixed $g\in R$, one can ask for the locus of primes $P\subseteq R$
for which the multiplication by $g$ map $\HH^i_I(R_P) \xrightarrow{g} \HH^i_I(R_P)$ is injective and the locus of primes for which this is surjective.
We show that these two loci are Zariski closed by describing explicitly the defining ideals of these loci, and we use
these to describe the defining ideal of the (Zariski closed) support for
$\HH^i_I(R/gR)$. We also extend the Zariski-closedness of $\HH^i_I(R/gR)$ to the case when $R$ has finitely many isolated singular points.

The methods used for the various calculations in this paper are described in section \ref{Section: Prime characteristic tools}.

\section{Prime characteristic tools}\label{Section: Prime characteristic tools}

\begin{defi}
Let $e\geq 0$. Let $T$ be a commutative ring of prime characteristic $p$.
\begin{enumerate}
  \item[(a)] Given any matrix (or vector) $A$ with entries in $T$, we define $A^{[p^e]}$ to be the matrix obtained from $A$ by raising its
entries to the $p^e$th power.

  \item[(b)] Given any submodule $K\subseteq T^\alpha$, we define $K^{[p^e]}$ to be the $R$-submodule of $T^\alpha$ generated by
$\{ v^{[p^e]} \,|\, v\in K \}$.

\end{enumerate}
\end{defi}

Henceforth in this section, $T$ will denote a regular ring with the property that $F_*^e T$ are \emph{intersection flat $T$-modules} for all $e\geq 0$, i.e.,
for any family of $T$-modules $\{M_\lambda\}_{\lambda\in\Lambda}$,
$$F_*^e T \otimes_T \bigcap_{\lambda\in\Lambda} M_\lambda= \bigcap_{\lambda\in\Lambda} F_*^e T \otimes_T  M_\lambda .$$
These include rings $T$ for which $F_*^e T$ are free $T$-modules (e.~g.~, polynomial rings and power series rings with $F$-finite coefficient rings,)
and also all complete regular rings (cf.~\cite[Proposition 5.3]{KatzmanParameterTestIdealOfCMRings}).
These rings have that property that for any collection of submodules $\{L_\lambda\}_{\lambda\in \Lambda}$ of $T^\alpha$,
$\left( \bigcap_{\lambda \in \Lambda} L_\lambda \right)^{[p^e]} =\bigcap_{\lambda \in \Lambda} L_\lambda^{[p^e]}$:
indeed, the regularity of $T$ implies that for any submodule $L\subseteq T^\alpha$, $L^{[p^e]}$ can be identified with $F_T^e( L)$ and
and the intersection-flatness of $F_*^e T$ implies
$$F_T^e ( \bigcap_{\lambda \in \Lambda} L_\lambda )= F_*^e T \otimes_T  \bigcap_{\lambda \in \Lambda} L_\lambda =
\bigcap_{\lambda \in \Lambda} F_*^e T \otimes_T  L_\lambda = \bigcap_{\lambda \in \Lambda} F_T^e ( L_\lambda ).
$$

The theorem below extends the $I_e(-)$ operation defined on ideals in \cite[Section 5]{KatzmanParameterTestIdealOfCMRings}
and in \cite[Definition 2.2]{BlickleMustataSmithDiscretenessAndRationalityOfFThresholds} (where it is denoted $(-)^{[1/p^e]}$) to submodules of free $R$-modules.

\begin{thm}
\label{qth root with respect to U}
Let $e\geq 1$.
Given a submodule $K\subseteq T^\alpha$ there exists a minimal submodule $L \subseteq T^\alpha$ for which
  $K\subseteq L^{[p^e]}$. We denote this minimal submodule $I_e (K)$.
\end{thm}
\begin{proof}
Let $L$ be the intersection of all submodules $M\subseteq T^\alpha$  for which $K\subseteq M^{[p^e]}$.
The intersection-flatness of $T$ implies that $K\subseteq L^{[p^e]}$ and clearly, $L$ is minimal with this property.

\end{proof}


When $F_*^e T$ is $T$-free,  this is
a straightforward generalization of the calculation of $I_e$ for ideals. To do so,  fix a free basis $\mathcal{B}$ for $F_*^e T$ and note that
every element $v\in T^\alpha$ can be expressed uniquely in the form $v=\sum_{b\in \mathcal{B}} u_{b}^{[p^e]} b$
where $u_{b}\in T^\alpha$ for all $b\in \mathcal{B}$.

\begin{prop} \label{Proposition: Computing Ie}
Let $e\geq 1$.
\begin{enumerate}
  \item [(a)] For any submodules $V_1, \dots, V_\ell\subseteq R^n$, $I_e(V_1 + \dots + V_\ell)=I_e(V_1)  + \dots + I_e(V_\ell)$.
  \item [(b)] Let $\mathcal{B}$ be a  free basis for $F_*^e T$.
  Let $v\in R^\alpha$ and let
  $$v=\sum_{b\in \mathcal{B}} u_{b}^{[p^e]} b $$
  be the unique expression for $v$ where $u_{b}\in T^\alpha$ for all
  $b\in \mathcal{B}$. Then  $I_e(T v)$ is the submodule $W$ of $T^\alpha$ generated by $\{ u_b  \,|\, b\in \mathcal{B} \}$.
  \end{enumerate}
\end{prop}

\begin{proof}
The proof of this proposition is a straightforward modification of the proofs of propositions 5.2 and 5.6 in \cite{KatzmanParameterTestIdealOfCMRings}
and Lemma 2.4 in \cite{BlickleMustataSmithDiscretenessAndRationalityOfFThresholds}.

Clearly, $I_e(V_1 + \dots + V_\ell)\supseteq I_e(V_i)$ for all $1\leq i\leq \ell$, hence $I_e(V_1 + \dots + V_\ell)\supseteq I_e(V_1)  + \dots + I_e(V_\ell)$.
On the other hand
$$(I_e(V_1)  + \dots + I_e(V_\ell))^{[p^e]} = I_e(V_1)^{[p^e]}  + \dots + I_e(V_\ell)^{[p^e]} \supseteq V_1+\dots+ V_\ell $$
and the minimality of $I_e(V_1 + \dots + V_\ell)$ implies that
$I_e(V_1 + \dots + V_\ell)\subseteq I_e(V_1)  + \dots + I_e(V_\ell)$ and (a) follows.

Clearly $v\in W^{[p^e]}$, and so $I_e(T v) \subseteq W$. On the other hand, let $W$ be a submodule of $T^\alpha$ such that $v\in W^{[p^e]}$.
Write $v = \sum_{i=1}^s r_i w_i^{[p^e]}$ for $r_i\in T$ and $w_i\in W$ for all $1\leq i \leq s$, and for each such $i$
write $r_i=\sum_{b\in \mathcal{B}} r_{b i}^{p^e} b$ where $r_{b i}\in T$ for all $b\in \mathcal{B}$.
Now
$$ \sum_{b\in \mathcal{B}} u_{b}^{[p^e]} b = v = \sum_{b\in \mathcal{B}} \left(\sum_{i=1}^s r_{b i}^{p^e} w_i^{[p^e]} \right) b $$
and since these are direct sums, we compare coefficients and obtain
$u_{b}^{[p^e]} = \left(\sum_{i=1}^s r_{b i}^{p^e} w_i^{[p^e]} \right)$ for all $b\in \mathcal{B}$
and so
$u_{b} = \left(\sum_{i=1}^s r_{b i} w_i \right)$ for all $b\in \mathcal{B}$
hence
$u_{b} \in W$ for all $b\in \mathcal{B}$.
\end{proof}

The behavior of the $I_e$ operation under localization and completion will be crucial for obtaining the results of this paper.
To investigate this we need the following generalization of \cite[Lemma 6.6]{LyubeznikSmithCommutationOfTestIdealWithLocalization}.

\begin{lem}\label{Lemma: commutation with localization and completion}
Let $\mathcal{T}$ be a completion of $T$ at a prime ideal $P$.
Let $\alpha\geq 0$ and let $W$ be a submodule of $\mathcal{T}^\alpha$.
For all $e\geq 0$, $W^{[p^e]} \cap T=(W\cap T)^{[p^e]}$.
\end{lem}

\begin{proof}
If $T$ is local with maximal ideal $P$, the result follows from a straightforward modification of the proof
of \cite[Lemma 6.6]{LyubeznikSmithCommutationOfTestIdealWithLocalization}.

We now reduce the general case to the previous case which implies that
$W^{[p^e]} \cap T_P=(W\cap T_P)^{[p^e]}$. Intersecting with $T$ now gives
$$W^{[p^e]} \cap T=(W\cap T_P)^{[p^e]}\cap T=(W\cap T_P\cap T)^{[p^e]}=(W\cap T)^{[p^e]} .$$
\end{proof}

\begin{lem}[cf.~\cite{Murru}]\label{Lemma: Ie and localization}
Let $\mathcal{T}$ be a localization of $T$ or a completion at a prime ideal.

For all $e\geq 1$, and all submodules $V\subseteq T^\alpha$, $I_e(V \otimes_T \mathcal{T})$ exists and equals $I_e(V)\otimes_T \mathcal{T}$.
\end{lem}

\begin{proof}
Let $L\subseteq \mathcal{T}^\alpha$ be a submodule, such that
$L^{[p^e]} \supseteq V \otimes_T \mathcal{T}$.
We clearly have
$L^{[p^e]} \cap T^\alpha= (L \cap T^\alpha)^{[p^e]}$ when
$\mathcal{T}$ is a localization of $T$ and when
$\mathcal{T}$ is a completion of $T$ this follows from
the previous Lemma.
We deduce that $(L \cap T^\alpha)\supseteq I_e (V \otimes_T \mathcal{T} \cap T^\alpha)$
and hence
$L\supseteq (L \cap T^\alpha) \otimes_T \mathcal{T} \supseteq I_e (V \otimes_T \mathcal{T} \cap T^\alpha) \otimes_T \mathcal{T}$.

But since $I_e (V \otimes_T \mathcal{T} \cap T^\alpha) \otimes_T \mathcal{T}$ satisfies
$$\left( I_e (V \otimes_T \mathcal{T} \cap ^\alpha) \otimes_T \mathcal{T}\right) ^{[p^e]}
=
I_e (V \otimes_T \mathcal{T} \cap T^\alpha)^{[p^e]} \otimes_T \mathcal{T}
\supseteq
(V \otimes_T \mathcal{T} \cap T^\alpha) \otimes_T \mathcal{T}
\supseteq
V \otimes_T \mathcal{T}$$
we deduce that  $I_e (V \otimes_T \mathcal{T} \cap T^\alpha) \otimes_T \mathcal{T}$
is the smallest submodule $K\subseteq \mathcal{T}^\alpha$ for which
$K^{[p^e]}\supseteq V \otimes_T \mathcal{T}$.
We conclude that $I_e(V \otimes_T \mathcal{T})$  equals
$I_e (V \otimes_T \mathcal{T} \cap T^\alpha) \otimes_T \mathcal{T}$.


We always have
$$I_e(V \otimes_T \mathcal{T}) = I_e (V \otimes_T \mathcal{T} \cap T^\alpha) \otimes_T \mathcal{T} \supseteq I_e(V)\otimes_T \mathcal{T}.$$
On the other hand
$$ \left( I_e(V)\otimes_T \mathcal{T} \right)^{[p^e]} =
I_e(V)^{[p^e]}\otimes_T \mathcal{T} \supseteq
V \otimes_T \mathcal{T} $$
hence
$I_e(V \otimes_T \mathcal{T})  \subseteq I_e(V)\otimes_T \mathcal{T}$
and thus
$I_e(V \otimes_T \mathcal{T})  = I_e(V)\otimes_T \mathcal{T}$.

\end{proof}

\section{Calculation of supports of $F_R$-finite $F_R$-modules}
\label{Section: Calculation of supports}

We begin by recalling the following result from \cite[Proposition 2.3]{LyubeznikFModulesApplicationsToLocalCohomology}.

\begin{rmk}[Vanishing of $\gamma_t$]
\label{rmk: vanishing of beta}
Let $\mathcal{M}$ be an $F_R$-finite $F_R$-module with a generating homomorphsim $\gamma :M\to F_R(M)$.
Let $\gamma_t$ denote the composition $M\to F_R(M)\to \cdots \to F^t_R(M)$.
We may assume that $M$ has a presentation: $R^{\alpha}\xrightarrow{A}R^\beta\to M\to 0$
and write the generating homomorphism as $\Coker(A)\xrightarrow{U}\Coker(A^{[p]})$
where $U$ is a $\beta \times \beta$ matrix with entries in $R$.
Then $\gamma_t$ is the composition of $\Coker(A)\to \cdots\to\Coker(A^{[p^t]})$.
Note that $\mathcal{M}=0$ if and only if there is a $t$ such that $\gamma_t=0$. We have
\begin{align}
\gamma_t=0 &\Leftrightarrow \Image(U^{[p^{t-1}]}\circ \cdots \circ U^{[p]}\circ U)\subseteq  \Image(A^{[p^t]})\notag\\
&\Leftrightarrow I_t(\Image U^{[p^{t-1}]}\circ \cdots \circ U^{[p]}\circ U)\subseteq \Image A\notag\\
& \Leftrightarrow I_1(U I_1(\cdots I_1(U I_1( \Image U)))\subseteq \Image A \notag \\
&\Leftrightarrow \frac{I_1(U I_1(\cdots I_1(U I_1(\Image U)))+\Image A}{\Image A}=0 \notag
 \end{align}
where we made repeated use of the facts that
for any submodule $M\subseteq R^\beta$ we have $I_{\ell+1}(M)=I_1(I_\ell(M))$
and also $I_\ell(U^{[p^\ell]}M)=U I_\ell(M)$.
\end{rmk}

\begin{thm}
\label{thm: uniform stabilization of Ie}
\hfill
\begin{itemize}
\item[(a)]
If
$$I_e(\Image U^{[p^{e-1}]}\circ \cdots \circ U^{[p]}\circ U)=I_{e+1}(\Image U^{[p^{e}]}\circ \cdots \circ U^{[p]}\circ U)$$
then
\begin{equation}
\label{uniform stabilization of Ie}
I_e(\Image U^{[p^{e-1}]}\circ \cdots \circ U^{[p]}\circ U)=I_{e+j}(\Image U^{[p^{e+j-1}]}\circ \cdots \circ U^{[p]}\circ U)
\end{equation}
for all $j\geq 0$.

\item[(b)]
There exists an integer $e$ such that (\ref{uniform stabilization of Ie}) holds.
\end{itemize}

\end{thm}

\begin{proof}
Write $V_e=I_e(\Image U^{[p^{e-1}]}\circ \cdots \circ U^{[p]}\circ U)$.
First we claim that if $V_e=V_{e+1}$ then
$V_e=V_{e+j}$ for all $j\geq 0$;
we proceed by induction on $j\geq 0$.
Using again the facts that for any submodule $W\subseteq R^\beta$ we have $I_{\ell+1}(M)=I_1(I_\ell(M))$
and that $I_\ell(U^{[p^\ell]}M)=U I_\ell(M)$, we deduce that,
if $j\geq 1$, then
$V_{e+j}=I_1 ( I_{e+(j-1)} (\Image U^{[p^{e+j-1}]}\circ \cdots \circ U^{[p]}\circ U))=I_1 (U V_{e+j-1})$
and this, by the induction hypothesis, equals  $I_1(U V_e)=V_{e+1}$.

Next, we wish to show that for each prime ideal $\mathfrak{p}$ there exists an integer $e_{\mathfrak{p}}$ such that
\begin{equation}
\label{local stabilization}
V_{e_{\mathfrak{p}}} R_{\mathfrak{p}}=V_{e_{\mathfrak{p}}+1} R_{\mathfrak{p}}.
\end{equation}
and that for this $e_{\mathfrak{p}}$, $V_{e_{\mathfrak{p}}+j} R_{\mathfrak{p}}=V_{e_{\mathfrak{p}}} R_{\mathfrak{p}}$ for all $j\geq 0$.
After completing at $\mathfrak{p}$, we have assume that our ring is a complete regular local ring, and we let $E$ denote the injective hull
of the residue field of $\widehat{R_\mathfrak{p}}$. As this ring is complete and regular, there is a natural Frobenius map on $E$ which we denote $T$, which
can be extended to a Frobenius map on direct sums of $E$ by letting $T$ act coordinate-wise; we denote these Frobenius maps also with $T$.


We now consider the Frobenius map $\Theta=U^t T$ on
$E^\beta$;
in \cite[Lemma 3.6]{KatzmanZhangAnnihilatorsOfArtinianModules}
it is shown that
$\Ann_{E^\beta} I_e(\Image U^{[p^{e-1}]}\circ \cdots \circ U^{[p]}\circ U R_\mathfrak{p})^t \subseteq {E^\beta}$
consists of all elements killed by $\Theta^e$.
Now (cf.~\cite[Proposition 1.11]{HartshorneSpeiserLocalCohomologyInCharacteristicP} and \cite[Proposition 4.4]{LyubeznikFModulesApplicationsToLocalCohomology})
show that there is an integer $e_{\mathfrak{p}}$ such that
$I_{e_{\mathfrak{p}}}(\Image U^{[p^{e-1}]}\circ \cdots \circ U^{[p]}\circ U R_\mathfrak{p})=I_{e_{\mathfrak{p}}+1} (\Image U^{[p^{e_{\mathfrak{p}}}]}\circ \cdots \circ U^{[p]}\circ U R_\mathfrak{p})$
and
$
I_{e_{\mathfrak{p}}} (
\Image U^{[ p^{e_{\mathfrak{p}}} -1 ]} \circ \cdots \circ U^{[p]}\circ U R_\mathfrak{p})=
I_{e_{\mathfrak{p}}+j} (\Image U^{[p^{e_{\mathfrak{p}}+j-1}]}\circ \cdots \circ U^{[p]}\circ U R_\mathfrak{p})
$
for all $j\geq 0$.
Crucially, Lemma \ref{Lemma: Ie and localization} implies that
$I_e\left(\left(\Image U^{[p^{e-1}]}\circ \cdots \circ U^{[p]}\circ U\right) R_\mathfrak{p}\right)=V_e R_\mathfrak{p}$
for all $e\geq e_{\mathfrak{p}}$ and so (\ref{local stabilization}) holds.

Consider the following subsets of $\Spec(R)$:
\[\mathcal{P}_t=
\{\mathfrak{p}\in \Spec(R)\mid V_t R_{\mathfrak{p}}=V_{t+1}R_{\mathfrak{p}}\}=
\Spec R \setminus \Supp \frac{V_t}{V_{t+1}}
.
\]
These form an increasing sequence of open subsets of $\Spec R$, and
since for each prime ideal $\mathfrak{p}$ there is an integer $t_{\mathfrak{p}}$ such that
\[V_{t_{\mathfrak{p}}}R_{\mathfrak{p}}=V_{t_{\mathfrak{p}}+1}R_{\mathfrak{p}},\]
we have $\bigcup_t\mathcal{P}_t=\Spec(R)$.
Now the quasicompactness of $\Spec R$, guarantees the existence of an integer $e$ such that $\mathcal{P}_e=\Spec(R)$;
clearly that $e$ satisfies (\ref{uniform stabilization of Ie}).
\end{proof}

\begin{cor}
\label{cor: support of F-modules}
If $I_e(\Image U^{[p^{e-1}]}\circ \cdots \circ U^{[p]}\circ U)=I_{e+1}(\Image U^{[p^{e}]}\circ \cdots \circ U^{[p]}\circ U)$, then
\[\Supp_R\Big(\frac{\Image I_e(U^{[p^{e-1}]}\circ \cdots \circ U^{[p]}\circ U)+\Image A}{\Image A}\Big)=\Supp_R(\mathcal{M}).\]
\end{cor}


\section{Our algorithm and its complexity}\label{Section: bounds on degrees}

Henceforth in this paper $R$ will denote a polynomial ring over a field $\mathbb{K}$ of prime characteristic $p$.

Let $\mathcal{M}$ be an $F$-finite $F$-module with a generating morphism $M\to F_R(M)$. Let $A$ be an $\alpha\times \beta$ matrix, which gives a presentation of $M$, {\it i.e.} $\Coker A \cong M$. Let $U$ be a $\beta\times \beta$ matrix,
for which the map $\Coker A \xrightarrow{U} \Coker A^{[p]}$ is isomorphic to a generating morphism $M\to F_R(M)$. We compute the support of $\mathcal{M}$ as follows.

\begin{enumerate}
\item[(1)] Initialize $L=R^\beta$.
\item[(2)] Compute $L^\prime=I_1(U L)$.
\item[(3)] Check whether $L^\prime\subseteq L$. If this holds, let $L=L^\prime$ and go to (2).
\item[(4)] Compute a presentation for $\frac{L+\Image A}{\Image A}$ as the cokernel of a matrix $W$ with entries in $R$.
\item[(5)] Compute the ideal $J$ of maximal minors of $W$.
\item[(6)] Output $J$: this is the defining ideal of the support of $\mathcal{M}$.
\end{enumerate}


In the rest of this section we discuss the complexity of this  algorithm for computing supports of $F$-finite $F$-modules.

We start with the following observation, relevant to the complexity of steps (2) and (3) above.
\begin{rmk}\label{Remark: degrees in U}
Let $\delta$ be the largest degree of an entry in $U$ and for any $j\geq 0$ let $\delta_j$ be the largest degree of a polynomial in a
generator of $L_j$. The calculation of $I_1(-)$ as described in Proposition \ref{Proposition: Computing Ie} implies that
$\delta_{j+1} \leq (\delta_j + \delta)/p$ hence
$$\delta_e\leq \frac{\delta_0}{p^e} +\delta(\frac{1}{p}+\dots +\frac{1}{p^e})\leq \frac{\delta}{p-1} .$$
\end{rmk}

The iterative calculation $I_1( U L)$ in step (2) involves
a matrix multiplication and a $I_1(-)$ operation, which, in view of Proposition \ref{Proposition: Computing Ie}, amounts to collecting terms in polynomials. The complexity of this step depends on
\begin{enumerate}
\item[(a)] the size $\beta$ of $U$, which is an input to the algorithm and does not depend on $p$,
\item[(b)] the total number of terms occurring in each of the coordinates of
a set of generators of $L$.
\end{enumerate}
In the worst case scenario, if the maximal degree of an entry in $U_p$ is $C p$,
the total number of terms in (b) is bounded by
$$\left(\binom{Cp+n-1}{n-1}\right)^\beta = \mathcal{O}\left(p^{\beta(n-1)}\right) .$$
In practice, the number of terms is much lower than this worst case.

Checking the inclusion in step (3) does {\it not} require computing Gr\"obner bases:
the discussion above shows that
both $L$ and $L^\prime$ are generated by vectors whose coordinates are polynomials whose degrees are bounded
by $D:=\delta/(p-1)$, i.e.,  $L$ and $L^\prime$ are given by generators in
$(R_{\leq D})^\beta$ where $R_{\leq D}$ denotes the $\mathbb{K}$-vector space
of polynomials of degrees at most $D$ in $R$.
Thus $L^\prime\subseteq L$ can be checked by
checking whether each given generator of $L^\prime$ is in the sub-vector space of $(R_{\leq D})^\beta$
spanned by the generators of $L$.

Computing the presentation in step (4) of the algorithm involves computing the syzygies of the generators of $L+\Image A$; this involves computing a Gr\"obner bases for $L+\Image A$.
\bigskip

In order to assess the practical advantage of our algorithm, we computed the support of 100
$F$-finite $F$-modules with randomly generated generating morphism $C \rightarrow F^1_R(C)$
where $C$ is a quotient of $R^2$ and
$R=\mathbb{Z}/2\mathbb{Z}[x_1,\dots, x_5]$. We denote $t_1$  the time in seconds required by our algorithm to compute the support and $t_2$ the
time in seconds required to compute a root using Grobner bases. The following is a plot of $\log t_2$ as a function of $\log t_1$
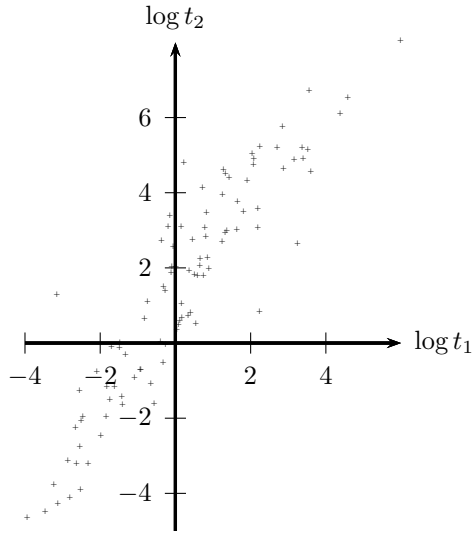
\begin{center}
\psset{xunit=0.5}
\psset{yunit=0.5}
\begin{pspicture}(-4,-5)(6,10)
{\small
\psaxes[linewidth=1.2pt,labels=all, ticks=all, Dx=2, Dy=2]{->}(0,0)(-4,-5)(6,8)[$\log t_1$,0][$\log t_2$,90]
}

\savedata{\mydata}[
{{1.63876950439933,3.01871610994437},
{4.38556785713775,6.12325348503967},
{-0.390791789724488,0.0439116167183247},
{-2.65099995977936,-2.23620763873092},
{2.09009954286735,4.91476295631074},
{1.24415171211024,2.70243446231657},
{-0.0587080845478252,2.56586432336792},
{3.24922266604005,2.65541565676376},
{1.37273281768035,3.00684529455756},
{2.19318643500312,3.5786515563923},
{-0.648613682199138,-1.06158821713568},
{0.831686636872985,3.47552223505735},
{3.15750667596024,4.88379385470143},
{0.153802047633545,3.11694497672688},
{0.647982436605553,2.06715891814531},
{0.859419806117312,2.2756927126878},
{-3.93809494385117,-4.62515148924587},
{1.65394271707644,3.77793602992461},
{3.55516804528747,6.71699662981193},
{0.659585226386918,2.25534462742122},
{-0.317220741562665,1.50743840890845},
{-1.69766275264524,-0.0977627849830652},
{1.91414965343612,4.34455968278192},
{-2.80667511858439,-4.09316725551968},
{-1.42366814090512,-1.4222780789999},
{0.544879163890842,0.519436444013886},
{-1.74603113052303,-1.50688132410901},
{0.163792618012069,1.04897376474328},
{0.229650336200333,4.81379281720142},
{1.28140595622246,4.62395268528746},
{-3.46273560818831,-4.4805925085032},
{3.37494580451777,5.19752971273534},
{0.0404411219975234,0.379784841281211},
{-0.913654210553025,-0.698593987517216},
{-1.61990247548198,-1.15467800560118},
{-1.8382223436384,-1.94130875174009},
{-2.54157351387009,-2.73627839399715},
{-0.8180077026416,0.671054934085129},
{0.719073178559188,4.1570313699962},
{-0.189724491018776,3.09614758180468},
{3.3991853589012,4.9196451030226},
{-1.32608394272091,-0.293186527539626},
{-1.81811781838759,-1.15554461042174},
{0.788366447140879,3.09451369435139},
{1.33628445336561,4.530321631979},
{-1.4809303682297,-0.118714844655577},
{-0.32315451249919,-0.512600531316229},
{-3.22917114022168,-3.74702988925611},
{1.42823492980274,4.40393244089479},
{0.172750903854005,0.693572090273026},
{-0.939134679299448,-0.708451699530806},
{-0.26698311495342,1.41813938292022},
{2.24237319849397,5.23225234757192},
{-1.47947161722407,0.0270994698817177},
{5.98682630708407,8.04793168672794},
{-3.12293629768994,-4.25158801975961},
{0.592707869221187,1.81521880131536},
{0.888541251821592,1.99220195683459},
{-0.56278405158645,-1.59493361023129},
{-2.86165602938803,-3.1094632797812},
{0.0106728420563039,2.03368421131705},
{4.58561624776622,6.53481302257947},
{-0.115259347837074,1.87874112355375},
{0.752575793324958,1.81140685571683},
{-2.55078494802623,-1.24302553158964},
{2.87711764221775,4.64625454608749},
{-2.50364430836473,-2.05407013513197},
{-0.0953045798200048,2.04045640077723},
{0.0832560010507104,0.506022105321668},
{-2.32001511714039,-3.19583091045063},
{-0.736290617543887,1.1110611457684},
{-0.260002056678472,-0.0301685280865741},
{3.59984056857179,4.56166438494545},
{2.70414926901726,5.20738558754233},
{0.811778745002822,2.84347812750242},
{2.04711202956428,5.05899361900479},
{-3.14676913966826,1.30339361480874},
{1.81521391723494,3.5012659194729},
{-2.46347123856223,-1.95575242574654},
{2.19719679917263,3.08098752612605},
{0.40074063199239,0.808910399982408},
{0.342383303221868,0.736958303182214},
{-1.40368063016253,-1.62625350536266},
{0.512757755995781,1.84214993952589},
{-1.97742661627311,-2.45213321419605},
{-2.51773939353208,-3.88524476439554},
{2.85225443578304,5.77595118670759},
{0.363774681053955,1.92804435916374},
{2.08231864883486,4.76643833358421},
{-0.142339210679783,3.40900019372518},
{3.5255837524185,5.16370539030327},
{2.23925331012155,0.846996386377052},
{-2.62997311630789,-3.19083030298035},
{0.118129642248045,0.609722092414818},
{-1.09005203243203,-0.897910684592725},
{1.25577272035907,3.96511627888098},
{-0.365926176361022,2.74000096233404},
{-2.08369055596925,-0.737594488595246},
{0.459536539895892,2.75956047194867},
{1.3223023572812,2.95602424168023},
}]
\dataplot[plotstyle=dots,showpoints=true, dotstyle=+, dotscale=0.5]{\mydata}


\end{pspicture}
\end{center}
This suggests that for this characteristic and rank, $t_2$ is approximately $t_1^2$.\footnote{The Macaulay2
code used to produce this data and the data itself is available at \cite{Fsupport}.
}

To further illustrate the effectiveness of our algorithm we compute the following example.

\begin{ex}
Consider three generic degree-2 polynomials in $t$: $F_1(t)=x_0+x_1t+x_2t^2$, $F_2(t)=y_0+y_1t+y_2t^2$, $F_1(t)=z_0+z_1t+z_2t^2$. For any two polynomials $F(t), G(t)$ let $\Res(F,G)$ denote their Sylvester resultant, e.~g.~,
\[\Res(F_1,F_2)=\det\begin{bmatrix}x_0 & x_1 &x_2 & 0\\ 0 & x_0 &x_1& x_2\\y_0 & y_1 &y_2 & 0\\ 0 & y_0 &y_1& y_2
\end{bmatrix}\]
Let $I$ denote the ideal generated by $\Res(F_1,F_2),\Res(F_1,F_3),\Res(F_2,F_3), \Res(F_1+F_2,F_3)$
in the polynomial ring $R$ over a field $k$ whose variables are the $x$,$y$, and $z$s above.
In \cite{LyubeznikMinimalResultantSystems} it was asked whether $\HH^4_I(R)=0$
and this was settled in prime characteristic $p>2$ (cf.~\cite{KatzmanCohomologicalDimensionResultantVariety})
and in characteristic zero (cf.~\cite[Theorem 3]{YanMinimalResultantSystems}.)
We used an implementation of our algorithm \cite{Fsupport} with Macaulay2 (\cite{Macaulay}) to settle the remaining case of characteristic $2$:
a 20-second run calculated the support of $\HH^4_I(R)$ to be empty.
\end{ex}

\section{The complexity of computing the support of local cohomology}
\label{Section: The complexity of computing the support of local cohomology}

We can use the fact that local cohomology modules
$\HH^i_J(R)$ are $F$-finite $F$-modules and apply the results in the previous sections to compute their supports.
To do so we would need to exhibit a generating morphism for these; a standard choice of generating morphism is given by the
$R$-linear map
$\Ext^j(R/J,R)\to \Ext^j(R/J^{[p]},R)$
induced by the natural surjection $R/J^{[p]} \to R/J$ (cf.~\cite[Proposition 1.11]{LyubeznikFModulesApplicationsToLocalCohomology}).
The calculation of this induced map would normally involve finding a free resolution $\mathbf{F}_\bullet$ for $R/J$,
extending the quotient map $R/J^{[p]} \to R/J$ to a map of free resulutions
$F_R^1 (\mathbf{F}_\bullet) \to \mathbf{F}_\bullet$,
applying $\Hom(-, R)$ to both resolutions, and finally computing the induced map of cohomologies.
Each of these steps involves computing multiple Gr\"obner bases, and, therefore, potentially unfeasible even in simple instances.
For example, even when $J$ is generated by one $f\in R$, the map of $\Ext^1(R/fR,R)\to \Ext^1(R/f^pR,R)$ above is
isomorphic to $R/f \xrightarrow[]{f^{p-1}} R/f^p$ and even when $f$ is, say, a random polynomial of degree 5 in 5 variables, expanding $f^{p-1}$
is not feasible beyond the first few primes (e.g., the M2 server Habanero crashes after $p=13$.)
Clearly, if the calculation of a generating morphism is infeasible, the algorithm of the preceding section cannot be used to calculate the support of 
the given local cohomology module.

\bigskip
However, in those cases when a generating morphism can actually be computed, the method above is guaranteed to produce generating morphisms which yield ``bounded'' generating morphism working over increasing prime characteristics $p$, as we describe below.
Let $S=\mathbb{Z}[x_1,\dots,x_n]$ and let $J\subseteq S$ be an ideal. For any prime $p$ let $J_p$ denote
the image of $J$ in $R_p=\mathbb{Z}/p\mathbb{Z}[x_1,\dots,x_n]$.
An interesting and natural question arising in this context is the description of the properties of the local cohomology
module $\HH^j_{J_p}(R_p)$ as $p$ ranges over all primes and we now turn our attention to these.

For different choices of prime $p$ the matrices $A$ and $U$ above will be different and this could result in different values of $\delta_e$ which are unbounded as $p$ ranges over all primes. We now show that this is not the case.
Let $U_p$  denote the
square matrix that induces the map $\Ext^j(R_p/J_p,R_p)\to \Ext^j(R_p/J^{[p]}_p,R_p)$ and  $\delta_p$ to denote the
maximal degree of entries in $U_p$. We also denote $L_{0,p}=R_p$ and $L_{i+1,p}=I_1(U_pL_{i,p})$ and use $\delta_{e,p}$
to denote the largest degree of a polynomial in a generator of $L_{e,p}$.

\begin{thm}
Let $0\to S^{b_s}\xrightarrow{A_s}\cdots \xrightarrow{A_2}S^{b_1}\xrightarrow{A_1}S\to S/J\to 0$ be a free resolution of $S/J$. Let $\Delta$ denote the maximal degree of any entry in $A_1, \dots, A_s$.
Let $p$ be a prime integer which is also a regular element on $S/J$.
Then $\delta_{e,p}\leq 2j\Delta$ for all integers $e\geq 1$.
\end{thm}
\begin{proof}
Since $p$ is a regular element on $S/J$, tensoring the free resolution of $S/J$ with $R/pR$ produces a free resolution of $R_p/J_p$. Hence the maximal degree of entries in the maps of this free resolution of $R_p/J_p$ is at most $\Delta$. Let $\theta_j$ denote the map $R_p^{b_j}\to R_p^{b_j}$ in the following commutative diagram induced by $R_p/J^{[p]}_p\to R_p/J_p$
\[\xymatrix{
0 \ar[r] & R_p^{b_s} \ar[r]^{A_{s,p}} & \cdots \ar[r] & R_p^{b_j} \ar[r]^{A_{j,p}} & \cdots \ar[r] & R_p^{b_1} \ar[r]^{A_{1,p}} & R_p \ar[r] & R_p/J_p \ar[r] & 0\\
0 \ar[r] & R_p^{b_s} \ar[r]^{A^{[p]}_{s,p}} \ar[u]^{\theta_s} & \cdots \ar[r] & R_p^{b_j} \ar[r]^{A^{[p]}_{j,p}} \ar[u]^{\theta_j} & \cdots \ar[r] & R_p^{b_1} \ar[r]^{A^{[p]}_{1,p}} & R_p \ar[u]^{=} \ar[r] & R_p/J^{[p]}_p \ar[r] \ar[u] & 0
}\]
An easy induction on $j$ shows that the maximal degree of entries in $\theta_j$ is at most $jp\Delta$.
The map $U_{p} : \Ext^j(R_p/J_p,R_p)\to \Ext^j(R_p/J^{[p]}_p,R_p)$ is induced by the transpose
of $\theta_j$ and hence the maximal degree in an entry of $U_p$ is also bounded by $jp\Delta$ (cf.~Remark \ref{Remark: degrees in U}.)
Now
\[\delta_{e,p}\leq \frac{jp\Delta}{p-1}\leq 2j\Delta,\]
for all $e\geq 1$.
\end{proof}

\begin{cor}
There is an integer $N$, independent of $p$, such that $\delta_{e,p}\leq N$ for all $e$ and all $p$.

In particular, there is an integer $N'$, independent of $p$, such that $\min\{e\mid L_{e,p}=L_{e+1,p}\}\leq N'$ for all $p$, i.e. for each prime integer $p$, the number of steps required to compute the stable value $L_{e,p}$ is bounded by $N'$.
\end{cor}
\begin{proof}
The second statement follows immediately from the first since, once the degree is bounded, the number of steps will be bounded by the number of monomials with the bonded degree.

To prove the first statement, it suffices to note there are only finitely many associated prime ideals of $S/J$ in $S$ and hence $p$ is a regular element on $S/J$ for almost all $p$.
\end{proof}



\section{Iterated local cohomology modules}

Let $f_1,\dots, f_m$ be a sequence of elements in $R$ and let $N$ be an $R$-module. We will write $K_i:=\bigoplus_{1\leq j_1<\cdots<j_i\leq m}N_{j_1\cdots j_i}$ to denote the $i$-th term of the Koszul (co)complex $\mathcal{K}^{\bullet}(M;\underline{f})$ (where each $N_{j_1\cdots j_i}=N$), and we will use $H^i(N;\underline{f})$ to denote the $i$-th Koszul (co)homology.

\begin{prop}
\label{prop: generating hom for local cohomology}
Let $\mathcal{M}$ be an $F_R$-finite $F_R$-module with a generating homomorphism $M\xrightarrow{\varphi}F_R(M)$ and let $I=(f_1,\dots,f_m)$ be an ideal of $R$. Then $H^i_I(\mathcal{M})$ admits a generating homomorphism
\[H^i(M;\underline{f})\to F_R(H^i(M;\underline{f})).\]
\end{prop}
\begin{proof}
Consider the following commutative diagram:
\[\xymatrix{
 &          \vdots & &                                                              \vdots & &                                                                  \vdots & \\
0\ar[r]  &  F_R(K^1)\ar[u]_{F_R(\phi_1)} \ar[r]_{F_R(\delta^1)}  &  \cdots \ar[r] &  F_R(K^i) \ar[u]_{F_R(\phi_i)} \ar[r]_{F_R(\delta^i)} &     \cdots \ar[r] &     F_R(K^m)=F_R(M) \ar[u]_{F_R(\phi_m)}\ar[r] & 0\\
0\ar[r]&    K^1\ar[u]_{\phi_1} \ar[r]_{\delta^1} &                  \cdots \ar[r]  & K^i \ar[u]_{\phi_i} \ar[r]_{\delta^i}  &                   \cdots \ar[r] &     K^m=M \ar[u]_{\phi_m}\ar[r] &
}
 \]

where the bottom row is the Koszul (co)complex of $M$ on $\underline{f}$ and
\[\phi_i:\bigoplus_{1\leq j_1<\cdots<j_i\leq m}M_{j_1\cdots j_i}\xrightarrow{\oplus_{1\leq j_1<\cdots<j_i\leq m}\varphi\circ (f_{j_1}\cdots f_{j_i})^{p-1}}F_R(\bigoplus_{1\leq j_1<\cdots<j_i\leq m}M_{j_1\cdots j_i}) .\]
It follows from \cite[1.10(c)]{LyubeznikFModulesApplicationsToLocalCohomology} that the
$\phi_i$
are generating morphisms of $\mathcal{M}_{f_{j_1}\cdots f_{j_i}}$. Therefore taking direct limit of each row of the diagram produces the \v{C}ech complex $\check{C}(\mathcal{M};\underline{f})$. Since taking direct limits preserves exactness, $\varinjlim(H^i(M;\underline{f})\to F_R(H^i(M;\underline{f}))\to \cdots)=H^i_I(\mathcal{M})$. Our conclusion follows.
\end{proof}

Combining what we have so far in this section, we now have an algorithm to compute the support of $H^{i_1}_{I_1}\cdots H^{i_s}_{I_s}(R)$. For example, the case $s=2$ relevant to the calculation of Lyubeznik numbers in handled as follows. Start with a generating morphism $\Ext^{i_2}(R/I_2,R)\to F_R(\Ext^{i_2}(R/I_2,R))$. Using Proposition \ref{prop: generating hom for local cohomology}, we know that the Koszul cohomology $H^{i_1}(\Ext^{i_2}(R/I_2,R);\underline{f})$ (with $I_1=(\underline{f})$) is a generating homomorphism of $H^{i_1}_{I_1}H^{i_2}_{I_2}(R)$. We may then apply Corollary \ref{cor: support of F-modules} to compute the support of $H^{i_1}_{I_1}H^{i_2}_{I_2}(R)$.


\section{The support of local cohomology of hypersurfaces}\label{Section: Local cohomology of hypersurfaces}

Throughout this section $R$ denotes a regular ring of prime characteristic $p$, $I\subseteq R$ an ideal, and $g\in R$ some fixed element.

Following \cite[\S 2]{LyubeznikFModulesApplicationsToLocalCohomology} we write
$$\HH^i_I(R)=\lim_{\rightarrow} \left[\Ext^i_R(R/I, R) \xrightarrow{\phi} F_R^1 \Ext^i_R(R/I, R) \xrightarrow{F_R^2 \phi}
F_R^2\Ext^i_R(R/I, R) \xrightarrow{F_R^3 \phi} \dots \right]$$
where $F_R^e(-)$ denotes the $e$th Frobenius functor, and  $\phi: \Ext^i_R(R/I, R) \xrightarrow{\phi} F_R^1 \Ext^i_R(R/I, R)\cong \Ext^i_R(R/I^{[p]}, R)$
is the $R$-linear map induced by the surjection $R/I^{[p]} \rightarrow R/I$.
For all $i\geq 0$ we fix a presentation
$R^{\alpha_i} \xrightarrow{ A_i } R^{\beta_i}$
where $A_i$ is a $\beta_i \times \alpha_i$ matrix with entries in $R$.
We can now find a $\beta_i \times \beta_i$ matrix $U_i$ with entries in $R$ for which the map $\phi: \Ext^i_R(R/I, R) \xrightarrow{\phi} F_R^1 \Ext^i_R(R/I, R)$
is isomorphic to the map $U_i : \Coker A_i \rightarrow F^1_R(\Coker A_i)=\Coker A_i^{[p]}$ given by multiplication by $U_i$.

\begin{thm}
\label{Theorem: injective surjective locus}
For any $i\geq 0$ consider the map $g: \HH^i_I(R_P) \rightarrow \HH^i_I(R_P)$ given by multiplication by $g$.
Let $\mathcal{I}^i$ denote the set of primes $P\subset R$ for which the map $g$
is not injective and let $\mathcal{S}^i$ denote the set of primes $P\subset R$ for which the map $g$ is not surjective.
For $\ell, e, j\geq 0$ write
$$V^{(\ell)}_{e j} =U_\ell^{[p^{e+j-1}]} U_\ell^{[p^{e+j-2}]} \cdots U_\ell^{[p^{e}]} .$$
Then
\begin{enumerate}
\item[(a)]
$\mathcal{I}^i$ is closed and equal to  $\displaystyle \Supp \frac{ (\ker V^{(i)}_{0 \eta} :_{R^\beta} g)}{\ker V^{(i)}_{0 \eta}}$ for some $\eta>0$,

\item[(b)]
$\mathcal{S}^i$ is closed and equal to $\displaystyle \Supp \frac{R^\beta}{ \bigcup_{j\geq 0} \left( gR^\beta+ \Image A^{[p^{j}]} :_{R^\beta} V^{(i)}_{0 j} \right)}$, and

\item[(c)]
the support of $\HH^i_I(R/ gR)$ is closed and equal to $\mathcal{I}^i \cup \mathcal{S}^i$.

\end{enumerate}
\end{thm}

\begin{proof}
Fix some $i\geq 0$ and write $\beta$, $A$ and $U$ for $\beta_i$, $A_i$ and $U_i$. The map $g: \HH^i_I(R) \rightarrow \HH^i_I(R)$ can be described as a map of direct limit systems
\begin{equation}\label{CD1}
\xymatrix{
\Coker A \ar@{>}[r]^{ U} \ar@{>}[d]^{ g} & \Coker A^{[p]} \ar@{>}[r]^{ U^{[p]}}  \ar@{>}[d]^{ g}  & \dots &  \ar@{>}[r]^{ U^{[p^{e-1}]}} & \Coker A^{[p^e]} \ar@{>}[r]^{ U^{[p^e]}}  \ar@{>}[d]^{ g} & \dots\\
\Coker A \ar@{>}[r]^{ U}  & \Coker A^{[p]} \ar@{>}[r]^{ U^{[p]}}   & \dots &  \ar@{>}[r]^{ U^{[p^{e-1}]}} & \Coker A^{[p^e]} \ar@{>}[r]^{ U^{[p^e]}}  & \dots\\
}.
\end{equation}
For any $ e , j\geq 0$ abbreviate  $V_{e j}=V^{(i)}_{e j}$, and note that it is the matrix corresponding to the
composition map $\Coker A^{[p^e]} \rightarrow \Coker A^{[p^{e+j}]}$ in the direct limits in (\ref{CD1}).
Any element in $\HH^i_I(R_P)$  can be represented by an element $a\in \Coker A^{[p^e]}_P$ for some $e\geq 0$, and this element
represents the zero element
if and only if there exists a $j\geq e$ for which $V_{e j}  a \in \Image A^{[p^{e+j}]}_P$, i.e.,
if and only if
$$ a \in ( \Image A^{[p^{e+j}]} :_{R^\beta} V_{e j} )_P.$$
Consider the kernels $K_j$ of the maps $V_{0 j} : \Coker A \rightarrow \Coker A^{[p^j]}$; these form an ascending chain of submodules of
$\Coker A$ and hence stabilize for all $j$ beyond some $\eta\geq 0$.
Note that the map $V_{e j} : \Coker A^{[p^e]} \rightarrow \Coker A^{[p^{e+j}]}$ is obtained by applying the exact functor $F_R^e(-)$ to the map
$V_{0 j}$, hence the kernels of the maps $V_{e j}$ also stabilize for $j\geq \eta$.

To prove (a) we now note that an element in $\HH^i_I(R_P)$ represented $a\in \Coker A^{[p^e]}_P$
is multiplied by $g$ to zero if and only if
$ a \in (\ker V_{e \eta} :_{R^{\beta}} g)_P$
and so $g$ is injective if and only iff
$\displaystyle \left(\frac{(\ker V_{e \eta} :_{R^{\beta}} g)} {\ker V_{e \eta}}\right)_P=0$,
i.e., if $g$ is not a zero divisor on $\left(R^\beta/  \ker V_{e \eta}\right)_P$.
But $R^\beta/ \ker V_{e \eta}=F_R^e (R^\beta/ \ker V_{0 \eta})$ and,
since $R$ is regular,
$F_R^e (R^\beta/ \ker V_{0 \eta})$ and $R^\beta/ \ker V_{0 \eta}$
have the same associated primes, so we deduce that multiplication by $g$ is injective if and only if
$g$ is not a zero divisor on $\left(R^\beta/ \ker V_{0 \eta}\right)_P$.
We deduce that for a prime $P\subset R$, multiplication by $g$ on $\HH^i_I(R_P)$ is injective
if and only if
$$ \left(\frac{ (\ker V_{0 \eta} :_{R^\beta} g)}{\ker V_{0 \eta}}\right)_P=0$$
so $\mathcal{I}^i=\Supp \frac{ (\ker V_{0 \eta} :_{R^\beta} g)}{\ker V_{0 \eta}}$.

To prove (b) we now note that an element in $\HH^i_I(R_P)$ represented $a\in \Coker A^{[p^e]}_P$
is in the image of $g$ if and only if there exists a $j\geq 0$ such that
$V_{e j} a \in \left(g R^\beta+ \Image A^{[p^{e+j}]}_P \right)$ hence $g$ is surjective if for all $e\geq 0$,
$$ \bigcup_{j\geq 0} \left( gR^\beta+ \Image A^{[p^{e+j}]} :_{R^\beta} V_{e j} \right)_P = {R^\beta}_P .$$
Furthermore,
\begin{eqnarray*}
\left(g R^\beta+ \Image A^{[p^{e+j}]} :_{R^\beta} V_{e j} \right)^{[p]} & = & \left(g^p R^\beta+ \Image A^{[p^{e+1+j}]} :_{R^\beta} V_{e+1,j} \right)\\
& \subseteq & \left(g R^\beta+ \Image A^{[p^{e+1+j}]} :_{R^\beta} V_{e+1,j} \right)\\
\end{eqnarray*}
so for for all $e\geq 0$,
$$\bigcup_{j\geq 0} \left( gR^\beta+ \Image A^{[p^{e+j}]} :_{R^\beta} V_{e j} \right)_P = {R^\beta}_P $$
if and only if
$$\bigcup_{j\geq 0} \left( gR^\beta+ \Image A^{[p^{j}]} :_{R^\beta} V_{0 j} \right)_P = {R^\beta}_P .$$
We conclude that $g$ is not surjective if and only if
$P\in \Supp R^\beta/ \bigcup_{j\geq 0} \left( gR^\beta+ \Image A^{[p^{j}]} :_{R^\beta} V_{0 j} \right)$.

To prove (c) consider the long exact sequence
$$ \dots \rightarrow \HH^i_I(R) \xrightarrow{g}  \HH^i_I(R) \rightarrow  \HH^i_I(R/gR) \rightarrow \HH^{i+1}_I(R) \xrightarrow{g}  \HH^{i+1}_I(R) \rightarrow \dots $$
induced by the short exact sequence $0 \rightarrow R \xrightarrow{g} R \rightarrow R/gR \rightarrow 0$.
Note that  $\HH^i_I(R/gR)_P=0$ if and only if both
$\left(\HH^i_I(R) \xrightarrow{g}  \HH^i_I(R)\right)_P$ is surjective and $\left(\HH^{i+1}_I(R) \xrightarrow{g}  \HH^{i+1}_I(R)\right)_P$
and the result follows.

\end{proof}

\begin{qu}
Theorem \ref{thm: uniform stabilization of Ie} gives us an effective method for the calculation of
$\mathcal{I}^i$. However, we do not know how to compute $\mathcal{S}^i$, hence we ask the following:
is there an effective method to bound the value of $e$ for which
$$\bigcup_{j\geq 0} \left( gR^\beta+ \Image A^{[p^{j}]} :_{R^\beta} V^{(i)}_{0 j} \right) =
\left( gR^\beta+ \Image A^{[p^{e}]} :_{R^\beta} V^{(i)}_{0 e} \right) ?$$
\end{qu}

It turns out that part of our Theorem \ref{Theorem: injective surjective locus} can be extended to the case of isolated singular points.

\begin{cor}
\label{cor: closedness iso sing}
Let $R$ be a noetherian commutative ring of prime characteristic that has finitely many isolated singular points. Let $g\in R$ be a nonzerodivisor. Then $\Supp(H^j_I(R/gR))$ is Zariski-closed for each integer $j$ and ideal $I$ of $R$.
\end{cor}
\begin{proof}
Let $\{\fm_1,\dots,\fm_t\}$ denotes the set of  isolated singular points of $R$. Set $\fa=\bigcap_{i=1}^t\fm_i$. Let $\{f_1,\dots,f_s\}$ be a set of generators of $\fa$. It follows from Theorem \ref{Theorem: injective surjective locus} that $\Supp_{R_{f_k}}(\HH^j_I(R_{f_k}/gR_{f_k}))$ is closed, {\it i.e.} it has finitely many minimal associated primes. By the bijection between the set of associated primes of $\HH^j_I(R/gR)$ that do not contain $f_k$ and $\Ass(\HH^j_I(R_{f_k}/gR_{f_k}))$, it follows that the minimal associated primes of $\HH^j_I(R/gR)$ are contained in the union of $\{\fm_1,\dots,\fm_t\}$ and the set of minimal associated primes of $\HH^j_I(R_{f_k}/gR_{f_k})$ which is a finite set.
\end{proof}

The proof of Corollary \ref{cor: closedness iso sing} can also be used to prove the following result which is of independent interest.

\begin{prop}
\label{thm: finiteness of ass isolated sing}
Let $R$ be either
\begin{enumerate}
\item a noetherian commutative ring of prime characteristic, or
\item of finite type over a field of characteristic 0.
\end{enumerate}
Suppose that $R$ has finitely many isolated singular points. Then $\HH^j_I(R)$ has only finitely many associated primes for each integer $j$ and each ideal $I$ of $R$.
\end{prop}
\begin{proof}
Let $\{\fm_1,\dots,\fm_t\}$ denotes the set of  isolated singular points of $R$. Set $\fa=\bigcap_{i=1}^t\fm_i$. Let $\{f_1,\dots,f_s\}$ be a set of generators of $\fa$. It follows from our assumptions on $R$ that $R_{f_k}$ is either a noetherian regular ring of prime characteristic or a regular ring of finite type over a field of characteristic 0
(cf.~\cite[Corollary 3.6]{LyubeznikFinitenessLocalCohomologyModules}).
Consequently, $\Ass(\HH^j_I(R_{f_k}))$ is finite for each generator $f_k$. Since there is a bijection between the set of associated primes of $\HH^j_I(R)$ that do not contain $f_k$ and $\Ass(\HH^j_I(R_{f_k}))$, it follows that
\[\Ass(\HH^j_I(R))\subseteq \bigcup_{k=1}^s\Ass(\HH^j_I(R_{f_k}))\bigcup \{\fm_1,\dots,\fm_t\}.\]
\end{proof}

\section*{Acknowledgement}
We would like to thank the anonymous referee whose suggestions and criticisms improved this paper substantially.

\bibliographystyle{skalpha}
\bibliography{KatzmanBib}

\end{document}